\documentclass[11pt]{article}
\usepackage[english]{babel}
\usepackage[cp1252]{inputenc}
\usepackage{graphicx}
\usepackage{color}
\usepackage{float}
\usepackage{amsmath, amssymb, amsthm}
\usepackage{indentfirst}
\newtheorem{Teo}{Theorem}[section]
\newtheorem{Def}[Teo]{Definition}
\newtheorem{Prop}[Teo]{Proposition}
\newtheorem{Lema}[Teo]{Lemma}
\newtheorem{Cor}[Teo]{Corollary}
\newtheorem{Conj}[Teo]{Conjecture}

\begin{document}

\title{\textsc{On a Conjecture by Kauffman on Alternative and Pseudoalternating Links}}

\author{Marithania Silvero \footnote{Partially supported by MTM2010-19355, MTM2013-44233-P, P09-FQM-5112 and FEDER.}\\ \\
Departamento de Álgebra.
Facultad de Matemáticas. \\
Universidad de Sevilla.
Spain.\\
{\tt marithania@us.es}\\ \\
}

\maketitle


\noindent \textbf{Abstract} \,
It is known that alternative links are pseudoalternating. In 1983 Louis Kauffman conjectured that both classes are identical. In this paper we prove that Kauffman Conjecture holds for those links whose first Betti number is at most 2. However, it is not true in general when this value increases, as we also prove by finding two counterexamples: a link and a knot whose first Betti numbers equal 3 and 4, respectively.

\bigskip

\noindent \textbf{Keywords:} Alternative links. Homogeneous links. Pseudoalternating links.
\vspace{1.2cm}
\mbox{ }


\section{Introduction}

In \cite{LibroKauffman} Louis Kauffman defined the family of alternative links, as an extension of the class of alternating links which preserves some of the nice properties of this well-known class. Previously, in 1976, E.J. Mayland and K. Murasugi introduced the class of pseudoalternating links \cite{Pseudoalternantes}. Alternative links are pseudoalternating, and Kauffman conjectured that the converse also holds:

\begin{Conj} {\rm{\cite{LibroKauffman}}} \label{Conjetura}
The classes of alternative and pseudoalternating links are identical.
\end{Conj}

Although this conjecture was stated by Kauffman, Mayland and Murasugi posed a similar question in \cite{Pseudoalternantes}. In this paper we prove Conjecture \ref{Conjetura} for links having their first Betti number, $\beta$, smaller than 3 (this includes the class of knots of genus one), and we provide counterexamples for links whose first Betti numbers equal 3 and 4, respectively (recall that $\beta(L) = \min \{rank(H_1(F)) | \mbox{ F is a Seifert surface for L}\}$). Namely, we present a genus two knot and a genus one two-components link which are pseudoalternating but not alternative.

\vspace{0.2cm}

The plan of the paper is as follows. In Section 2 we recall the definitions of alternative and pseudoalternating links; we also recall the definition of homogeneous links, an intermediate family introduced by Peter Cromwell in \cite{CromwellHom}. In Section 3 we disprove Kauffman Conjecture by finding a link and a knot being pseudoalternating and non-homogeneous, hence non-alternative: $L9n18\{1\}$ and $10_{145}$, whose first Betti numbers equal 3 and 4, respectively. Finally, Section 4 is devoted to prove that the Conjecture holds in the case of links having first Betti number smaller than 3, providing an alternative proof of the characterization of homogeneous genus one knots given in \cite{Pedro}; we also give an upper bound for the number of primitive flat surfaces plumbed to construct a generalized one spanning a given pseudoalternating link.

\vspace{1.1cm}

\noindent \textbf{Acknowledgements} \,
I want to thank Pedro M. G. Manchón and Juan González-Meneses for telling me about Kauffman Conjecture, which constituted the starting point of this work. I am also grateful for their numerous valuable comments, their suggestions and corrections on preliminary versions of this paper. I would also like to thank Józef H. Przytycki for sharing some of his intuitive ideas with me, and to the anonymous referee for his/her many interesting suggestions on an earlier version of this paper.\\

\section{Alternative, Homogeneous and Pseudoalternating Links} \label{defalternpseudo}

As the alternative, homogeneous and pseudoalternating characters of a link are orientation dependant, from now on all links will be oriented and non-split.

\vspace{0.2cm}

Given an oriented diagram $D$ of a link $L$, it is possible to smooth every crossing coherently with the orientation of the diagram. After doing this for all crossings in $D$, we obtain a set of topological (Seifert) circles. Following Kauffman, the {\emph{spaces}} of the diagram $D$ are the connected components of the complement of its Seifert circles in $S^2$, as opposed to the regions of the knot diagram. Draw an edge joining two Seifert circles at the place where there was a crossing in $D$, and label the edge with the sign of the corresponding crossing ($+$ or $-$). We will refer to the resulting set of topological circles and labeled edges as the Seifert diagram of $D$, because of the analogy of this process to Seifert's algorithm for constructing an orientable surface spanning a link.

\begin{Def} {\rm{\cite{LibroKauffman}}}
An oriented diagram $D$ is alternative if all the edges in any given space of $D$ have the same sign. An oriented link is alternative if it admits an alternative diagram.
\end{Def}

An oriented diagram is alternating if and only if it is alternative and the sign of the edges in its Seifert diagram changes alternatively when passing through adjacent spaces \cite[Lemma 9.2]{LibroKauffman}. There are nevertheless alternative links which are not alternating: for instance, every positive (hence alternative) non-alternating link, like the knot $8_{19}$.

\vspace{0.2cm}

We consider now the family of homogeneous links, introduced by Peter Cromwell in 1989 \cite{CromwellHom}. From the Seifert diagram associated to $D$, we can construct a graph $G_D$ as follows: associate a vertex to each Seifert circle and draw an edge connecting two vertices in $G_D$ for each edge joining the associated circles in the Seifert diagram; each edge must be labeled with the sign $+$ or $-$ of its associated crossing in $D$. The signed graph $G_D$ is called the Seifert graph associated to $D$. Note that $G_D$ can be obtanined from the Seifert diagram of $D$ by collapsing each circle to a vertex.

\vspace{0.2cm}

Given a connected graph $G$, a vertex $v$ is a \emph{cut vertex} if $G \backslash \{v\}$ is disconnected. A block of $G$ is a maximal subgraph of $G$ with no cut vertices. Blocks of the graph $G$ can be thought of in the following way: remove all the cut vertices of $G$; each remaining connected component together with its adjacent cut vertices is a block of $G$.

\begin{Def} {\rm{\cite{CromwellHom}}}
A Seifert graph is homogeneous if all the edges of a block have the same sign, for all blocks in the graph. An oriented diagram $D$ is homogeneous if its associated Seifert graph $G_{D}$ is homogeneous. An oriented link is homogeneous if it admits a homogeneous diagram.
\end{Def}

\begin{figure}[t] \label{943}
\centering
\includegraphics[width = 10cm]{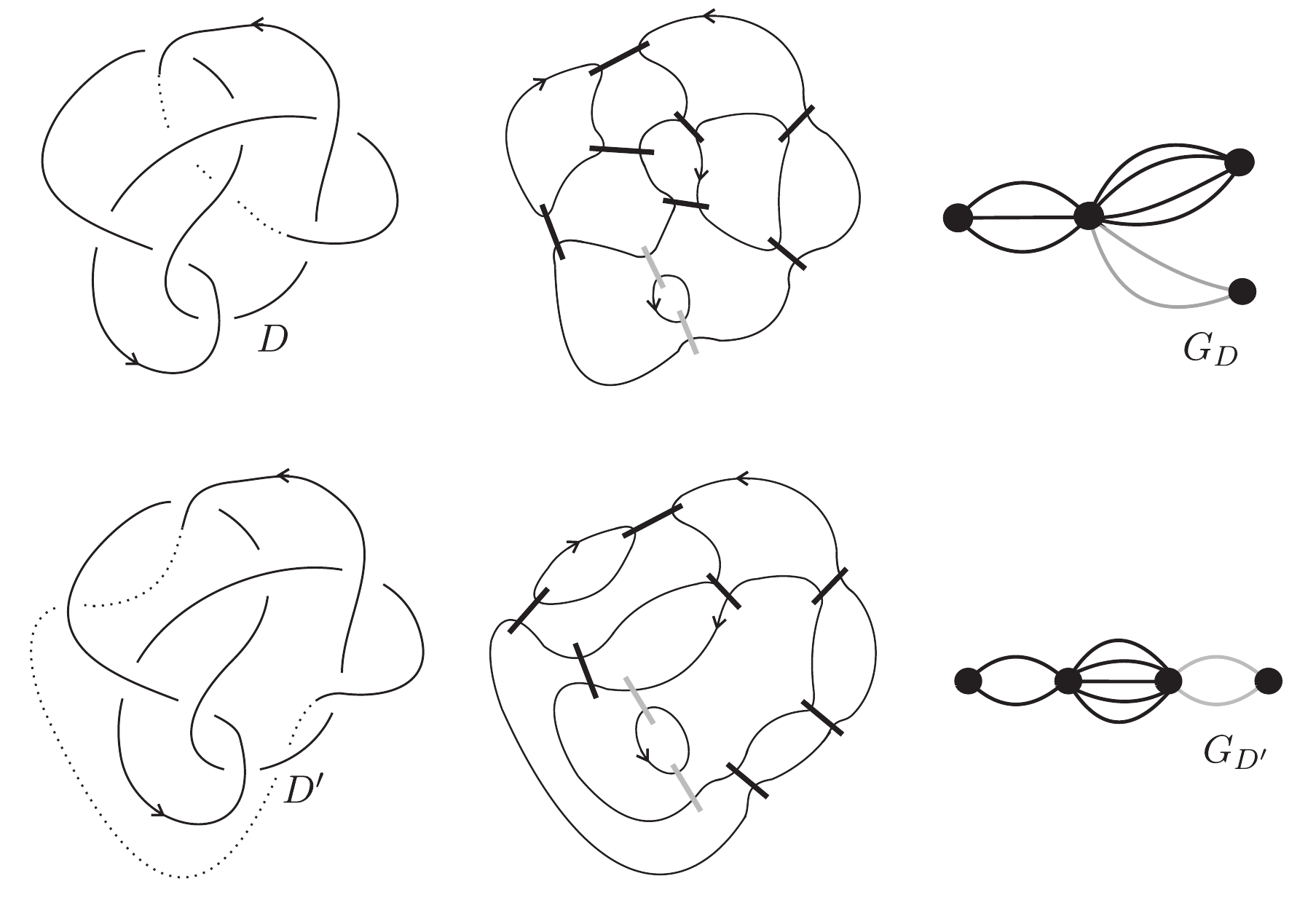}
\caption{\small{Two diagrams $D$ and $D'$ of the knot $9_{43}$, their associated Seifert diagrams and Seifert graphs. A dark edge has a positive label; a light edge a negative one. $D$ shows that $9_{43}$ is homogeneous, as the diagram in the figure is so; $D$ is non-alternative, but since $D'$ is so, $9_{43}$ is alternative.}}
\label{943}
\end{figure}

Note that the original diagram $D$ can be recovered from its Seifert diagram, as the sign and position of the crossings in the diagram are preserved (see Figure \ref{943}). However, as the relative position of the circles and the order of the edges is not encoded in the Seifert graph, $D$ cannot be recovered from $G_D$.

\vspace{0.2cm}

Let us finally introduce pseudoalternating links. Starting from an oriented diagram $D$ of a link $L$, the surface $S_D$ obtained by applying Seifert's algorithm \cite{LibroCromwell} is known as the canonical surface (called projection surface in \cite{LibroCromwell}) of $L$ associated to $D$. The graph $G_D$ can also be thought as the spine graph of the corresponding canonical surface.

\vspace{0.2cm}

{\emph{Primitive flat surfaces}} \cite{Pseudoalternantes} are those canonical surfaces arising from positive or negative diagrams whose Seifert diagrams have no nested circles. A {\emph{generalized flat surface}} is, roughly speaking, an orientable surface obtained by gluing a finite number of primitive flat surfaces along some of their discs.

\vspace{0.2cm}

More precisely, given two primitive flat surfaces $S_1$ and $S_2$, choose a disc of each one, $d_1$ and $d_2$. Now, identify both discs in such a way that there exists a sphere $S^2 \subset S^3$ separating $S_3$ into two non empty 3-balls $B_1$ and $B_2$ such that $S_i \subset B_i$ and $S^2 \cap S_i = d_i$, for $i = 1,2$. Bands starting at $d_1$ and $d_2$ are not allowed to overlap when identifying $d_1$ and $d_2$. This special kind of Stallings plumbing (or Murasugi sum) will be noted by $*$ (see Figure \ref{pegadopseudo}). Generalized flat surfaces are obtained as a finite iteration of this process, plumbing a primitive flat surface in each step.

\vspace{0.2cm}

The first Betti number of a surface $S$, $\beta(S)$, is the rank of its first homology group. For a primitive flat surface this is just the number of holes in the surface, or equivalently the number of connected components of the complement of its spine graph in the plane minus 1. As the Euler characteristic of one of these surfaces can be computed as its number of discs $d_S$ minus its number of bands $b_S$, it follows that $\beta(S) = b_S - d_S + 1$. Notice that the first Betti number is additive under this special kind of plumbing: each time one plumbs two surfaces both plumbing discs are identified, so the resulting spine graph can be thought as gluing the two original graphs along a vertex.

\vspace{0.2cm}

\begin{figure} \label{pegadopseudo}
\centering
\includegraphics[width = 12cm]{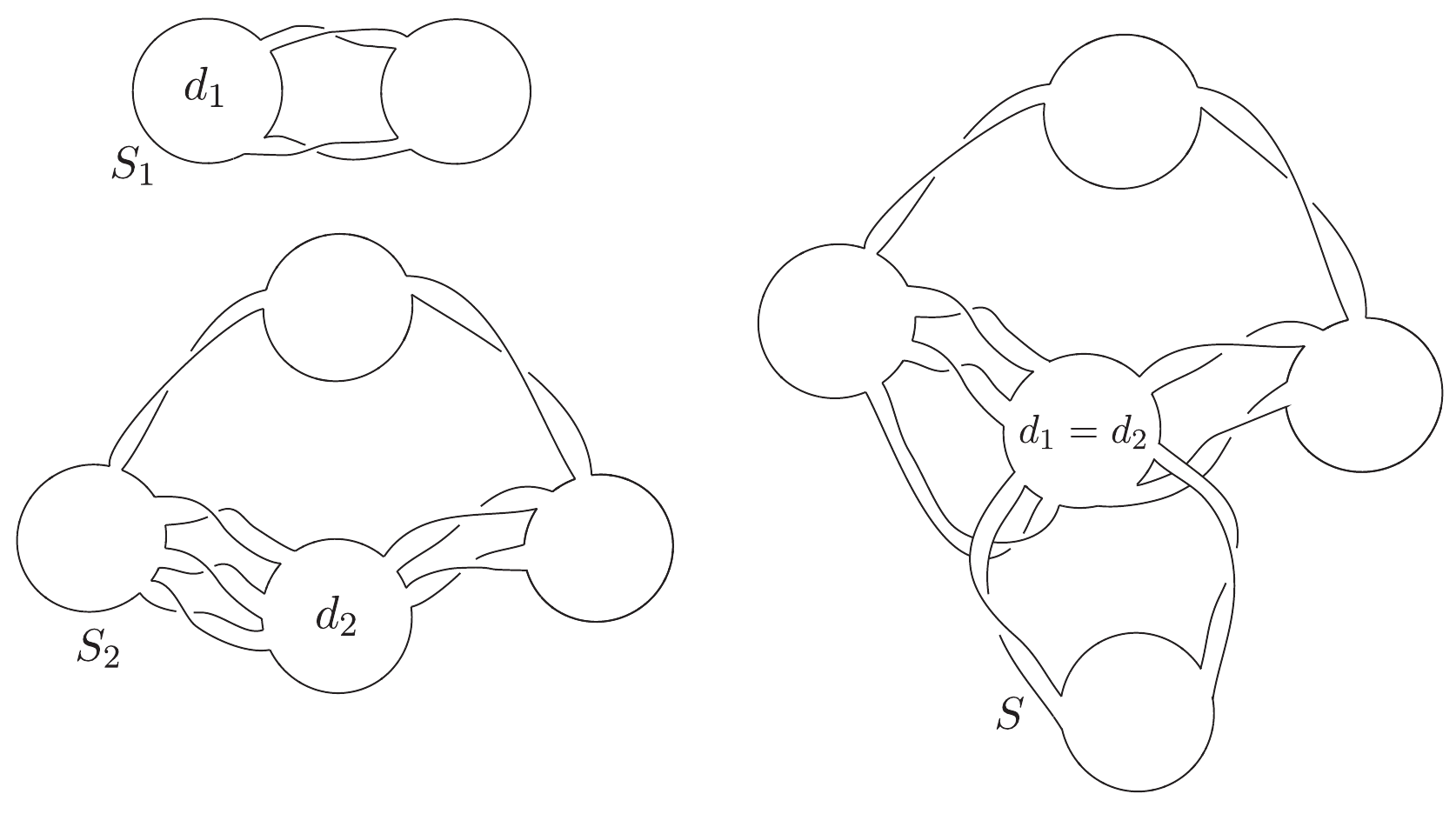}
\caption{\small{$S_1$ and $S_2$ are primitive flat surfaces, with $\beta(S_1) = 1$ and $\beta(S_2) = 4$. By an identification of discs $d_1$ and $d_2$ one obtains the generalized flat surface $S = S_1 * S_2$, having $\beta(S) = 5$. The link spanned by $S$ is a pseudoalternating link.}}
\end{figure}

As generalized flat surfaces are orientable, the following definition makes sense:

\begin{Def} {\rm{\cite{Pseudoalternantes}}}
An oriented link is said to be pseudoalternating if it is the boundary of a generalized flat surface, with the natural inherited orientation.
\end{Def}

The following result implies that the first Betti number of a pseudoalternating link $L$, $\beta(L)$, is given by any generalized flat surface spanning it:

\begin{Prop}\label{propaux}
Among all the connected Seifert surfaces that span a given pseudoalternating link, those being generalized flat surfaces have maximal Euler characteristic, or equivalently, minimal genus.
\end{Prop}

For a proof see {\cite[Proposition 4.4]{Pseudoalternantes}}. It also follows from a well-known result by Gabai \cite[Corollary 6.7]{Gabai}.

\vspace{0.2cm}

Given the Seifert diagram of a link, circles containing other circles in both sides (inside and outside) become cut vertices in the associated Seifert graph. As a result, an alternative diagram is homogeneous, so an alternative link is homogeneous.

\vspace{0.2cm}

Let $D$ be a homogeneous diagram and $S_D$ its projection surface. Let $G_1, \ldots, G_n$ be the blocks of the homogeneous Seifert graph $G_D$, $D_i$ the subdiagram of $D$ (together with some arcs) associated to the subgraph $G_i \subset G_D$, and $S_{D_i} \subset S_D$ the projection surface constructed from $D_i$. Then $S_D = S_{D_1} * \ldots * S_{D_n}$. Since blocks in a graph do not contain cut vertices, each $S_{D_i}$ is a primitive flat surface: all bands are twisted in the same way and the discs on the surface are either not nested, or there exists a single disc containing the other ones (and this situation is isotopic to the previous one). Hence, $S_D$ is a generalized flat surface. This proves that homogeneous links are pseudoalternating, so Conjecture \ref{Conjetura} can be restated by saying that the classes of alternative, homogeneous and pseudoalternating links are equal.

\vspace{0.2cm}

In the two following sections we will present two pseudoalternating links which are not homogeneous, hence non-alternative; these links are counterexamples to Conjecture \ref{Conjetura}.

\section{Two counterexamples to Kauffman Conjecture}

The main problem when trying to deal with Kauffman Conjecture is that identifying whether a link is alternative or pseudoalternating is not easy. Of course, by finding an alternative diagram one shows the alternativity of a link, but this does not help when the link is not alternative. In this sense, working with the family of homogeneous links will help us.

\vspace{0.2cm}

Cromwell's paper \cite{CromwellHom} is devoted to the study of homogeneous links; he provides some sufficient conditions for determining the non-homogeneous character of a link. In particular, we will use the following result:

\begin{Teo} {\rm{\cite[Corollary 5.1]{CromwellHom}}} \label{Corol}
If $L$ is a homogeneous link and the leading coefficient of its Conway polynomial $\nabla(L)$ is $\pm 1$, then the crossing number of $L$ is at most $2 \cdot \mbox{maxdeg } \nabla(L)$.
\end{Teo}

The Conway polynomial of $L$ is defined as $\nabla(L) = P (1,z)$, where $P(v,z)$ is the \mbox{HOMFLYPT} polynomial defined by the skein relation $v^{-1} P (L_+) - vP(L_-) = z P(L_0)$, with normalization $P(\textnormal{unknot}) = 1$. Here $L_+$, $L_-$ and $L_0$ are links represented by diagrams $D_+$, $D_-$ and $D_0$ respectively, which only differ in the neighborhood of a point as shown in Figure \ref{suavizado}.

\begin{figure}[H] \label{suavizado}
\centering
\includegraphics[width = 5cm]{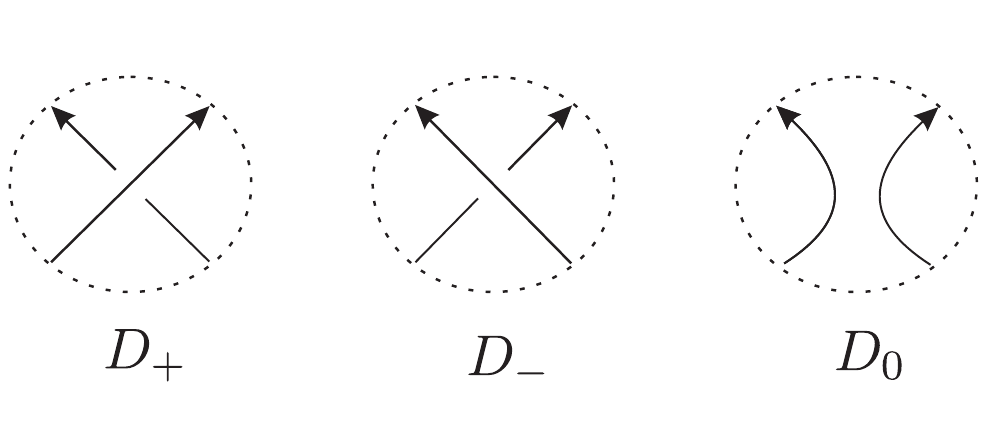}
\caption{\small{Diagrams $D_+$, $D_-$ and $D_0$.}}
\end{figure}

\vspace{0.2cm}

We are ready now to prove the following result:

\begin{Teo} \label{Teoenlaces}
There exists a pseudoalternating link which is not homogeneous.
\end{Teo}

\begin{proof}
Consider the oriented link $L$ with two components presented by the oriented diagram $D$ shown in Figure \ref{L9n181}. This link is $L9n18\{1\}$ in \cite{Linkinfo}, which corresponds to the two-components link $L9n18$ in Thistlethwaite table (or $9^2_{53}$ in Rolfsen table). Its Conway polynomial is $\nabla (L) = z^3 + 4z$, so by Theorem \ref{Corol}, if $L$ were homogeneous its crossing number would be at most $2 \cdot 3 = 6$, yielding a contradiction. Consequently, $L$ is non-homogeneous.

\begin{figure} \label{L9n181}
\centering
\includegraphics[width = 11.8cm]{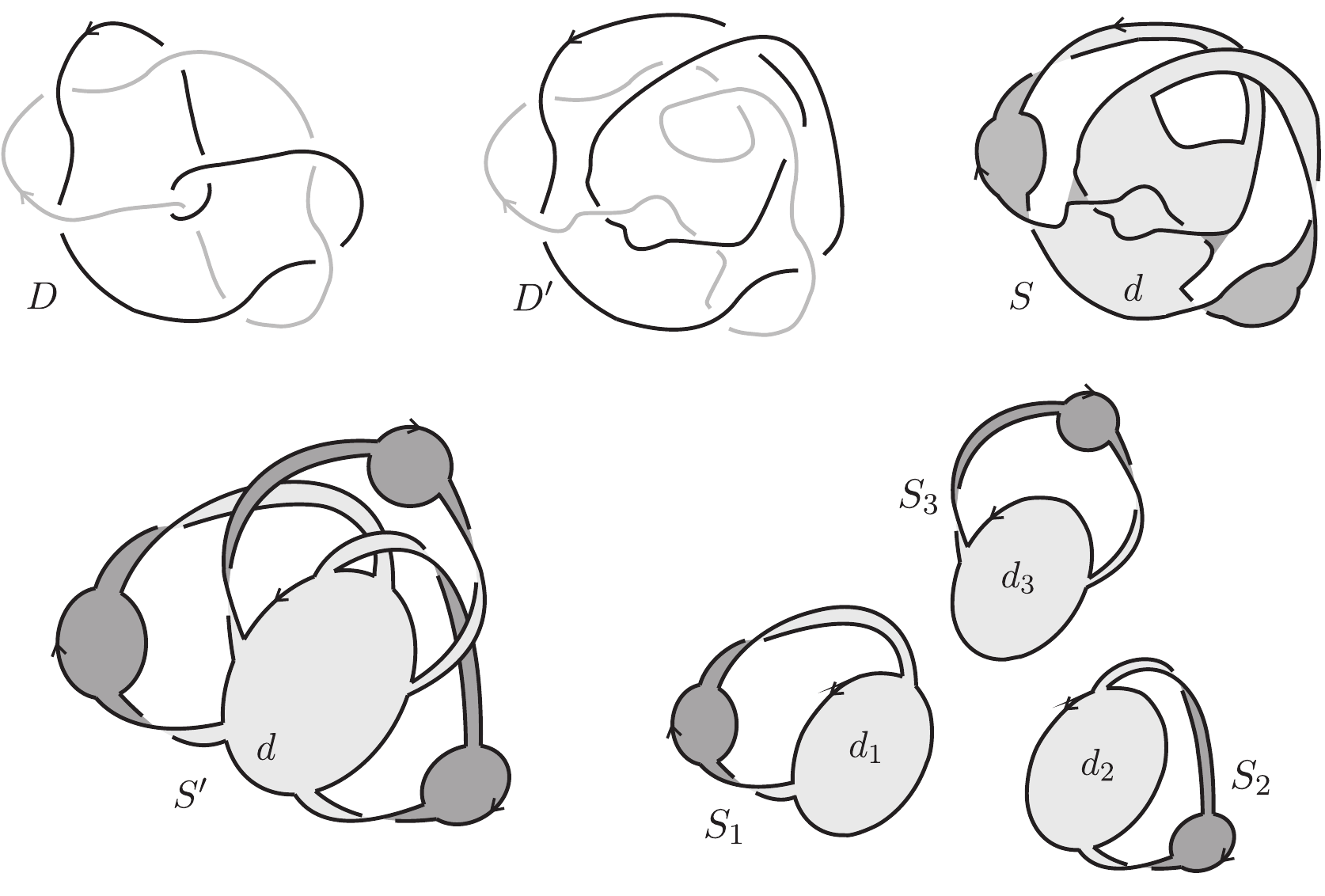}
\caption{\small{$D$ and $D'$ are diagrams of the link $L = L9n18\{1\}$ in \cite{Linkinfo}; $S$ is a Seifert surface for $L$, and $S'$ is obtained from $S$ just by overturning the subsurface $S_3$ over the disc $d$. Finally $S' = S_1*S_2*S_3$.}} \label{L9n181}
\end{figure}

\vspace{0.2cm}

See Figure \ref{L9n181}. The diagram $D$ can be transformed by a Reidemeister III and a Reidemeister I move into the diagram $D'$. The new diagram $D'$ allows us to see $L$ as the boundary of a certain surface $S$. The surface $S$ is clearly isotopic to $S'$, and $S'$ is the result of performing two Stallings plumbings of three surfaces, $S_1, S_2, S_3$, using discs $d_1$, $d_2$ and $d_3$ as ``gluing patches'', this is, $S = S_1 * S_2 * S_3$. Each of these surfaces consists on two discs joined by a pair of bands twisted in a positive way, so $S_1, S_2$ and $S_3$ are primitive flat surfaces; hence $S$ is a generalized flat surface and $L$ is a pseudoalternating link.
\end{proof}

Since every alternative link is homogeneous, the link $L9n18\{1\}$, with $\beta(L) = 3$, is a counterexample to Conjecture \ref{Conjetura}.

\begin{Cor}
There exists a pseudoalternating link which is not alternative.
\end{Cor}

Note that if one changes the orientation of one component of $L$ (for example, reversing the orientation of the dark component in Figure \ref{L9n181}) the resulting link is negative, hence alternative.

\vspace{0.2cm}

At this point one can wonder if there exist knots or links which are pseudoalternating and non-alternative with any possible orientation of its components. We will show such an example by finding a knot of genus two with these properties, as pseudoalternating, homogeneous and alternative characters are not orientation-dependant in the case of knots. This knot, whose first Betti number equals 4, is pseudoalternating and non-alternative. It would be interesting to find a link with more than one component being a counterexample to the conjecture with all possible orientations.

\begin{Teo} \label{Teonudos}
There exists a pseudoalternating knot which is not homogeneous.
\end{Teo}

\begin{proof}
The proof is analogous to that of Theorem~\ref{Teoenlaces}. Let $K$ be the genus two knot $10_{145}$ in Rolfsen table, with the orientation given in Figure \ref{10145}. Its Conway polynomial is $\nabla (K) = z^4 + 5z^2 + 1$, so by Theorem \ref{Corol} we deduce that $K$ is not homogeneous, as $10 > 2 \cdot 4$.

\vspace{0.2cm}

\begin{figure}[h] \label{10145}
\centering
\includegraphics[width = 12.1cm]{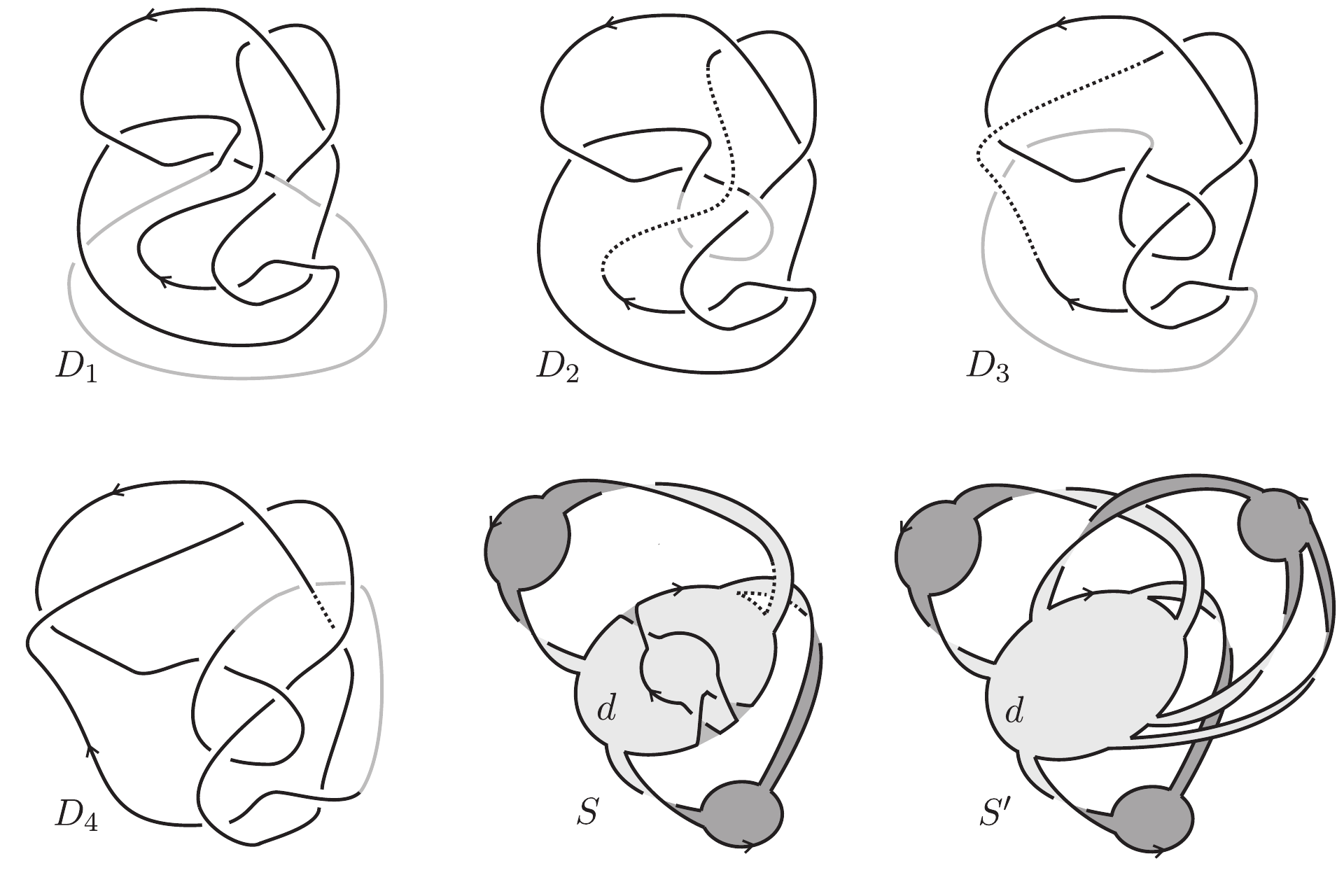}
\caption{\small{From $D_1$ to $D_2$ and from $D_3$ to $D_4$ just move the grey strand, leaving unchanged the rest of the diagram; from $D_2$ to $D_3$ perform two Reidemeister III moves on the dotted strand. By performing a Reidemeister I move on the dotted strand in $D_4$, we see that $S$ is a Seifert surface for $K$; $S'$ is the result of overturning one of the primitive flat surfaces over the disc $d$.}} \label{10145}
\end{figure}

See Figure \ref{10145}. By a finite sequence of Reidemeister moves, the classical diagram representing $K$ in \cite{Knotinfo}, $D_1$, can be transformed into $D_4$ by performing the following steps: from $D_1$, we obtain $D_2$ by leaving unchanged the diagram except for the grey strand; perform two Reidemeister III moves on the dotted strand of $D_2$ in order to get $D_3$; finally, transform $D_3$ into $D_4$ by moving the grey strand. At this point it is easy to see that the surface $S$ bounds $K$; $S$ is a generalized flat surface obtained by performing two Stallings plumbing of three primitive flat surfaces (two of them consist on a pair of discs joined by two bands, and the other one consists on two discs together with three bands, as shown in $S'$) using twice the same disc, $d$, as ``gluing patch''. As a result, $K$ is a pseudoalternating knot.
\end{proof}

\begin{Cor} \label{nudopseudonoalt}
There exists a pseudoalternating knot which is not alternative.
\end{Cor}

Since $10_{145}$ is a non-alternating, non-positive knot of genus two, the following result by Stoimenov provides an alternative proof of Corollary \ref{nudopseudonoalt}. We remark that the definition of homogeneous link given by Stoimenov in \cite{Stoimenov} is equivalent to our definition of alternative link, taken from Kauffman \cite{LibroKauffman}.

\begin{Teo} {\rm{\cite[Theorem 4.1]{Stoimenov}}}
Any alternative genus two knot $K$ is alternating or positive.
\end{Teo}

As we said before, alternativity, homogeneity and pseudoalternation depend on the orientation in the case of links but not when working with knots. Moreover, a knot is alternative, homogeneous or pseudoalternating if and only if its mirror image is so. Since $K = 10_{145}$ oriented as in Figure \ref{10145} is chiral, its mirror image $K^*$ is another counterexample to Conjecture \ref{Conjetura}.

\section{The case $\beta(L) \leq 2$}

Theorems \ref{Teoenlaces} and \ref{Teonudos} show that Conjecture \ref{Conjetura} does not hold for links or knots in general. On the contrary, in this section we prove that all links whose first Betti number is smaller than 3 satisfy Kauffman Conjecture (including the particular case of knots of genus one).

\vspace{0.2cm}

Given a pseudoalternating link $L$, it is not easy to find a diagram which helps to find a generalized flat surface bounding the link. Here we give an upper bound for the number of non-trivial (that is, non-isotopic to a disc) primitive flat surfaces which can be plumbed in order to get a generalized flat surface bounded by $L$.

\begin{Lema} \label{cotasupprim}
Let $S$ be a generalized flat surface spanning a pseudoalternating link $L$ with first Betti number $\beta(L)$. If $S = S_1 * S_2 * \ldots * S_n$, with each $S_i$ a non-trivial primitive flat surface, then $n \leq \beta(L)$.
\end{Lema}

\begin{proof}
As $S$ is a generalized flat surface, $\beta(S) = \beta(L)$, by Proposition \ref{propaux}. As the surface $S_i$ is connected and non-trivial, $\beta(S_i) \geq 1$. Then, after performing $n$ Stallings plumbings, one gets $\beta(S_1*S_2* \ldots * S_n) = \beta(S_1) + \beta(S_2) + \ldots + \beta(S_n) \geq n$.

\end{proof}

\begin{Lema} \label{lema}
Let $S$ be a generalized flat surface spanning a pseudoalternating link $L$. If $S$ is either a primitive flat surface or a generalized flat surface constructed as the Stallings plumbing of two primitive flat surfaces, then $L$ is alternative.
\end{Lema}

\begin{proof}
If $S$ is a primitive flat surface, then $L$ is a positive or negative link, hence it is alternative.

\vspace{0.2cm}

Otherwise $S = S_1 * S_2$, with $S_1$ and $S_2$ two primitive flat surfaces. Now just turn $S_2$ over (or under) the ``gluing disc'' (see Figure \ref{volcados} for an example of a plumbing of two annuli; notice that any primitive flat surface could be turned over in the same way). Then it is clear that the projection of the boundary of the new surface on the plane that contains the discs of $S_1$ provides a diagram for $L$ with just two spaces containing edges; this diagram is alternative, as edges related to $S_1$ and $S_2$ are in different spaces.
\end{proof}

\begin{Cor}\label{beta2}
Every pseudoalternating link $L$ with $\beta(L) \leq 2$ is alternative.
\end{Cor}

\begin{figure}[t]
\centering
\includegraphics[width = 15.5cm]{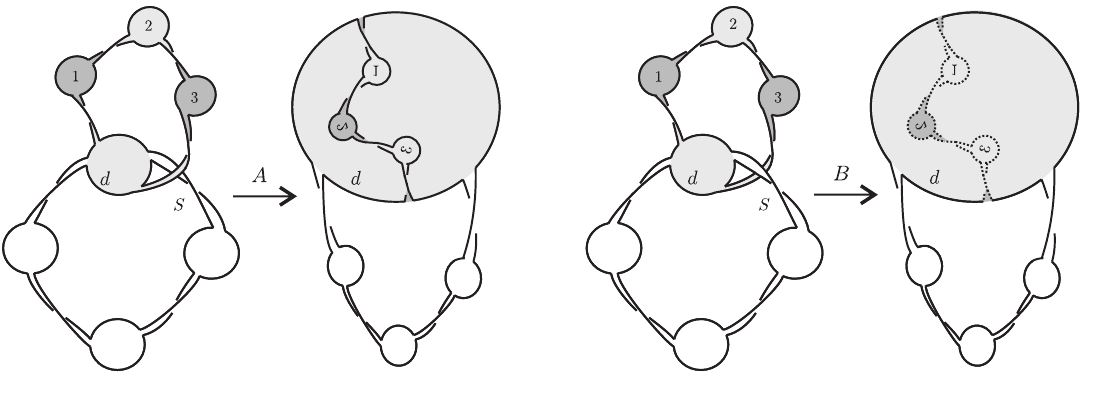}
\caption{\small{$S = S_1 * S_2$, with $S_1$ and $S_2$ two primitive flat surfaces plumbed by using $d$ as gluing disc; $S_2$ has been colored. In $A$ ($B$), the surface $S_2$ has been overturned over (under) the disc $d$.}}
\label{volcados}
\end{figure}

The following Corollaries are particular cases of Corollary \ref{beta2}:

\begin{Cor}\label{kauffmanverdadG0}
Every pseudoalternating genus zero link with three or less components is alternative.
\end{Cor}

\begin{proof}
Let $L$ be a pseudoalternating genus zero link with $\mu$ components, and let $S$ be a generalized flat surface whose boundary is $L$. As $\beta(L) = 2g(L) + \mu - 1$, by Lemma \ref{cotasupprim} the surface $S$ is plumbing of at most $\mu - 1$ non-trivial primitive flat surfaces. If $\mu = 1$, $L$ is the trivial knot, which bounds a disc, hence it is alternative. If $\mu$ is $2$ or $3$, the result holds by Lemma \ref{lema}.
\end{proof}

We claim that pseudoalternating genus zero links with four or five components are also alternative (notice that their first Betti numbers are 3 and 4, respectively). In order to show this, one must check all different possible ways of gluing primitive flat surfaces to obtain such links, and take into account that, when plumbing a surface $S$ with $\beta(S) = 1$ to a pseudoalternating surface $S'$, the number of boundary components increases (decreases) by one when both bands coming from the gluing disc are attached to the same (different) boundary component of $S'$. This case by case procedure is straightforward but lengthy, so we prefer to omit it.

\vspace{0.2cm}

Another consequence of Lemmas \ref{cotasupprim} and \ref{lema} is the following corollary:

\begin{Cor}\label{kauffmanverdadK1}
Every pseudoalternating genus one knot is alternative.
\end{Cor}

\begin{proof}
Let $K$ be a pseudoalternating genus one knot, and let $S$ be a generalized flat surface whose boundary is $K$. By Lemma \ref{cotasupprim}, as $\beta(K) = 2 g(K) + \mu - 1  = 2$, $S$ is plumbing of at most two non-trivial primitive surfaces. Lemma \ref{lema} completes the proof.
\end{proof}

We have shown that a knot of genus one is pseudoalternating if and only if it is homogeneous. As a consequence, we obtain an alternative proof of the following result:

\begin{Teo} {\rm{\cite{Pedro}}}
A genus one knot is homogeneous if and only if it belongs to one of the two following classes of knots: \\
1.- Pretzel knots with diagram $P(a,b,c)$, where $a,b,c$ are odd integers with the same sign. \\
2.- Pretzel knots with diagram $P(m, \, e, \, \stackrel{k}{\ldots} \, , \, e)$, where m and k are non-zero even integers and $e = \pm 1$.
\end{Teo}

\begin{proof}
As before, $K = \partial S$, where $S$ is a primitive flat surface with $\beta(S) = 2$ or $S = S_1 * S_2$ is the plumbing of two primitive flat surfaces having $\beta(S_1) = \beta(S_2) = 1$.

\vspace{0.2cm}

In the first case, $K$ is a positive or negative knot. As $S$ is connected and $\beta(S) = 2$, $S$ must be as shown in Figure \ref{pretzelpares} (left). Write $A, B, C$ for each of the three subsurfaces consisting on ``a linear path of bands and discs'' in $S$ ($A, B$ and $C$ in cyclic order when traveling through the two common discs $d_\delta$ and $d_\gamma$). As $K$ is a knot, each of these paths must start and end in discs with different orientations (otherwise the link bounding the surface would be a 3-components link). Hence, their respective numbers of bands, $a', b', c'$, are odd. Let $\varepsilon$ be their common sign. $K$ is the Pretzel knot $P(\varepsilon \cdot a', \, \varepsilon \cdot b',  \, \varepsilon \cdot c')$, as can be seen in Figure \ref{pretzelpares}.

\vspace{0.2cm}

Now assume that $S = S_1 * S_2$. Note that after the plumbing, the pair of bands attached to $d_2$ in $S_2$ must alternate with the two bands attached to $d_1$ in $S_1$ in the gluing disc $d = d_1 = d_2$; otherwise, the resulting surface would span a 3-components link. Let $b_i$ be the number of bands in $S_i$ and $\varepsilon_i$ their signs, \, $i = 1,2$. $\beta(S_1) = \beta(S_2) = 1$, so both $S_1$ and $S_2$ are twisted Hopf-bands, that is, each of them consists on $k$ discs and $k$ bands joined forming a circle. Since they are oriented, $b_i$ is even and it follows (see an example in Figure \ref{pretzelimpares}) that $K$ is the Pretzel knot $P(\varepsilon_1 \cdot b_1, \, \varepsilon_2, \, \stackrel{b_2}{\ldots} \, , \, \varepsilon_2)$. Write $m = \varepsilon_1 \cdot b_1$, \, $k = b_2$ and $e = \varepsilon_2$.
\end{proof}

\begin{figure}[h]
\centering
\includegraphics[width = 12.5cm]{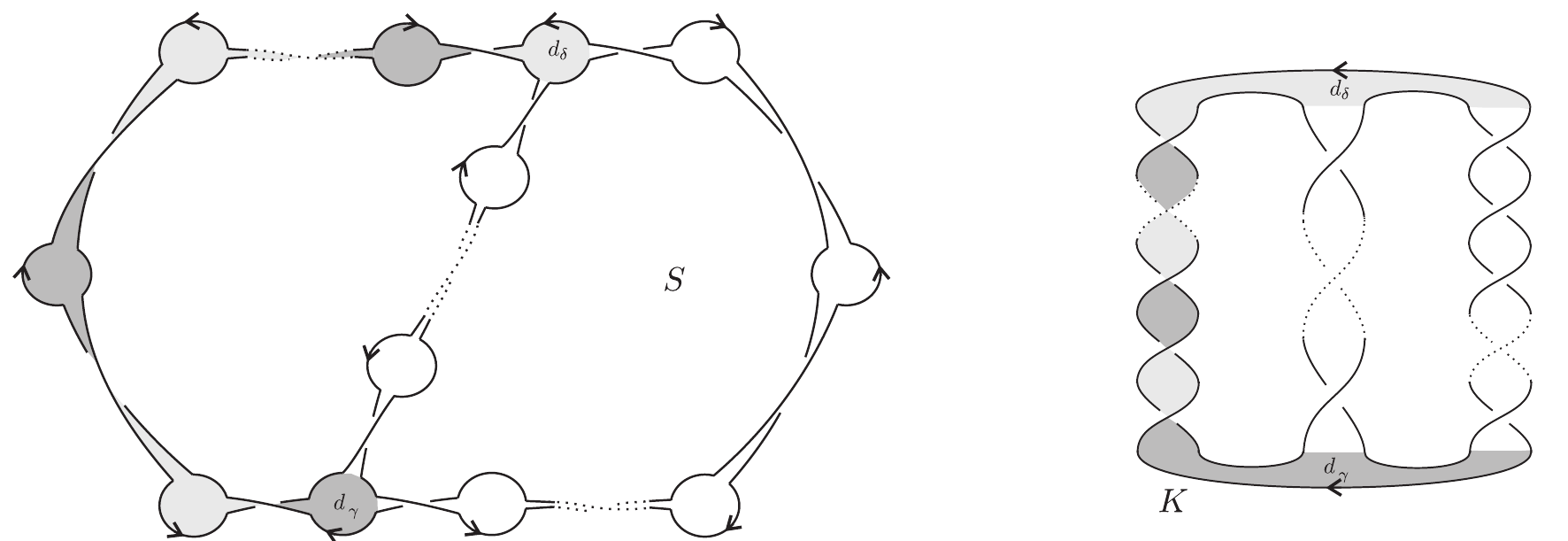}
\caption{\small{$S$ is a primitive flat surface; an even number of discs and bands can be attached in the place of the dotted lines (or removed). One of the three subsurfaces, say $A$, has been colored. If one thinks on the example in the picture as if all lines were non-dotted, then $a' = 5$, $b' = 3$, $c' = 5$ and $\varepsilon = -$.}}
\label{pretzelpares}
\end{figure}

\begin{figure}[h]
\centering
\includegraphics[width = 15.5cm]{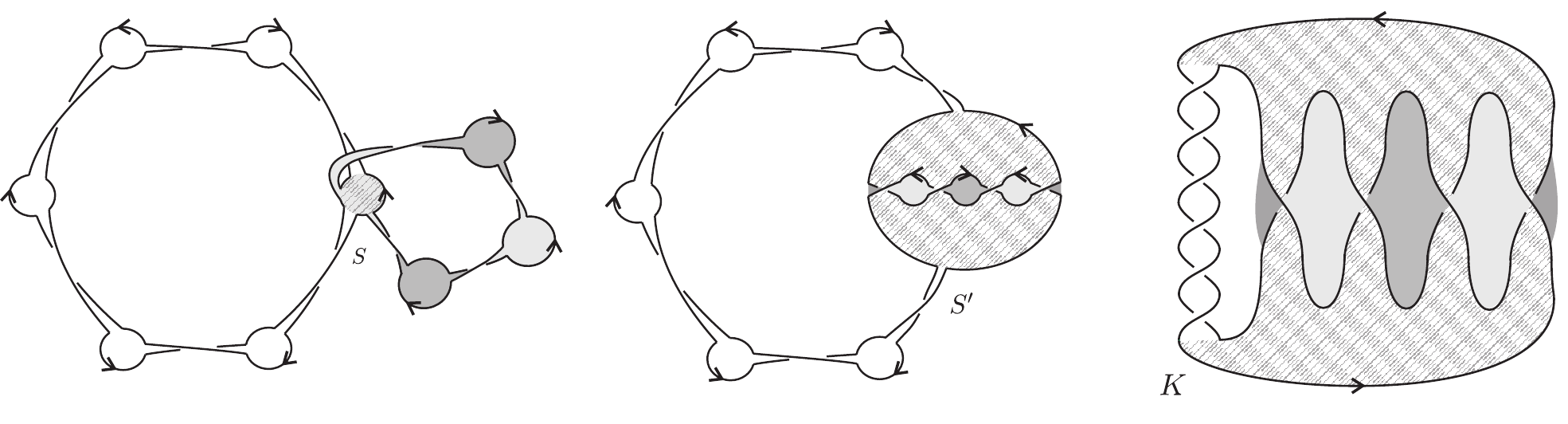}
\caption{\small{$S= S_1 * S_2$ is a generalized flat surface; $S_2$ has been colored. After overturning $S_2$ over the dotted disc, it is clear that $K$ is the Pretzel knot $P = (-6, 1, 1, 1, 1)$}}
\label{pretzelimpares}
\end{figure}

\bibliographystyle{plain}
\bibliography{Bibliograf}

\end{document}